\documentclass[psamsfonts]{amsart}
\usepackage{amssymb,amsfonts}
\usepackage{enumerate}
\usepackage{mathrsfs}
\usepackage{url}

\newtheorem{thm}{Theorem}[section]

\newtheorem{prop}[thm]{Proposition}
\newtheorem{lem}[thm]{Lemma}
\newtheorem{conj}[thm]{Conjecture}

\theoremstyle{definition}
\newtheorem{defn}[thm]{Definition}

\theoremstyle{remark}

\makeatletter
\let\c@equation\c@thm
\makeatother
\numberwithin{equation}{section}

\title{Numbers with integer complexity close to the lower bound}
\author{Harry Altman and Joshua Zelinsky}
\date{July 16, 2012}

\begin{document}

\newcommand{\cpx}[1]{\|#1\|}
\newcommand{\dft}{\delta}

\newcommand{\N}{{\mathbb N}}
\newcommand{\R}{{\mathbb R}}
\newcommand{\Z}{{\mathbb Z}}
\newcommand{\Q}{{\mathbb Q}}

\begin{abstract}
Define $\cpx{n}$ to be the \emph{complexity} of $n$, the smallest
number of $1$'s needed to write $n$ using an arbitrary combination of
addition and multiplication.
John Selfridge showed that $\cpx{n}\ge 3\log_3 n$
for all $n$.  Define the \emph{defect} of $n$, denoted $\dft(n)$, to be
$\cpx{n}-3\log_3 n$; in this paper we present a method for classifying all $n$
with $\dft(n)<r$ for a given $r$.  From this, we derive several
consequences.  We prove that $\cpx{2^m 3^k}=2m+3k$ for $m\le 21$ with $m$ and
$k$ not both zero, and present a method that can, with more computation,
potentially prove the same for larger $m$.  Furthermore, defining $A_r(x)$ to be
the number of $n$ with $\dft(n)<r$ and $n\le x$, we prove that
$A_r(x)=\Theta_r((\log x)^{\lfloor r \rfloor+1})$, allowing us to conclude that
the values of $\cpx{n}-3\log_3 n$ can be arbitrarily large.
\end{abstract}

\maketitle

\section{Introduction}

The complexity of a natural number $n$ is the least number of $1$'s needed to
write it using any combination of addition and multiplication, with the order of
the operations specified using  parentheses grouped in any legal nesting.  For
instance, $11$ has complexity of $8$, since it can be written using $8$ ones as
$(1+1+1)(1+1+1)+1+1$, but not with any fewer.  This notion was introduced by
Kurt Mahler and Jan Popken in 1953 \cite{MP}.  It was later circulated by
Richard Guy \cite{Guy}, who includes it as problem F26 in his \emph{Unsolved
Problems in Number Theory} \cite{UPINT}. It has since been studied by a number
of authors, e.g. Rawsthorne \cite{Raws} and especially Juan Arias de Reyna
\cite{Arias}.

Following Arias de Reyna \cite{Arias} we will
denote the complexity of $n$ by $\cpx{n}$. Notice that for any
natural numbers $n$ and $m$ we will have
\begin{equation*}
\cpx{1}=1, \quad
\cpx{n+m}\le \cpx{n}+\cpx{m},\quad
\cpx{nm}\le \cpx{n}+\cpx{m},\quad
\end{equation*}

More specifically, for any $n>1$, we have
\begin{displaymath}
\cpx{n}=\min_{\substack{a,b<n\in \mathbb{N} \\ a+b=n\ \mathrm{or}\ ab=n}}
	\cpx{a}+\cpx{b}.
\end{displaymath}
This fact together with $\cpx{1}=1$ allows one to compute $\cpx{n}$ recursively.
If  the equality $\cpx{n} = \cpx{a}+ \cpx{b}$ holds, with either $n=a+b$ or
$n=ab$, then we will say $n$ can be written \emph{most-efficiently} as $a+b$ or
as $ab$, respectively.

Integer complexity is approximately logarithmic; it satisfies the bounds
\begin{equation*}
3\log_3 n\le \cpx{n} \le 3\log_2 n,\qquad n>1.
\end{equation*}
The upper bound can be obtained by writing $n$ in binary and finding a
representation using Horner's algorithm. The lower bound follows from results
described below. The lower bound is known to be attained infinitely often,
namely for all $n=3^k$. The constant in the upper bound above can be improved further \cite{upbds}, 
and it is an open problem to determine the true
asymptotic order of magnitude of the upper bound. At present even the
possibility that an asymptotic formula $\cpx{n} \sim 3 \log_3 n$ might hold has
not been ruled out.

Let $E(k)$ be the largest number writable with $k$ ones, i.e., with complexity
at most $k$.
John Selfridge (see \cite{Guy}) proved that $E(1) =1$, and the larger values
depend on the residue class of $k$ modulo $3$, namely for $k=3j +i \ge 2$,
\begin{eqnarray*}
E(3j) &=&3^j\\
E(3j+1) &=& 4 \cdot 3^{j-1} \\
E(3j+2) &= & 2 \cdot 3^j
\end{eqnarray*}
Observe that $E(k)\le 3^{k/3}$ in all cases, and that equality holds for
cases where $3$ divides $k$.
These formulas also show that
$E(k) > E(k-1)$, a fact that implies that the integer 
$E(k)$ requires exactly
$k$ ones. This yields the following result:
\begin{thm} \label{th1}
For  $a=0, 1,2$ and for all $k \ge 0$ with $a+k \ge 1$, one has
\begin{equation*}\label{3m}
\cpx{2^a \cdot3^k}=2a +3k.
\end{equation*}
\end{thm}

Further results are known on the largest possible integers having a given
complexity. We can generalize the notion of $E(k)$ with the following
definition:

\begin{defn}
Define $E_r(k)$ to be the $(r+1)$-th largest number writable using $k$ ones,
i.e.~
complexity at most $k$, so long as there are indeed $r+1$ or more distinct such
numbers. Thus $E_r(k)$ is defined only for $ k \ge k(r)$.
Here $E_0(k)=E(k)$.
\end{defn}

Daniel A. Rawsthorne \cite{Raws} determined a formula for $E_1(k)$, namely:
\begin{equation*}
E_1(k)=\frac{8}{9} E(k), \qquad k\ge 8
\end{equation*}
Direct computation establishes that
$E_1(k)\le(8/9)E(k)$ holds for all $2 \le k \le 7$ (note that $E_1(1)$ is not
defined). From this fact we deduce that, for $0\le a \le 5$ and all $k \ge 0$
with $a+k>0$,
$$
\cpx{2^a \cdot 3^k}=2a+3k.
$$
J. Iraids et al. \cite{data2}
has verified that $\cpx{2^a 3^k}=2a+3k$ for
$2 \le 2^a \cdot 3^k \le 10^{12}$ , so in particular
$$\cpx{2^a}=2a, \quad \mbox{ for} \quad 1\le a\le 39.$$
These results together with results given later in this paper lend
support to the following conjecture, which was originally formulated
as a question in Guy \cite{Guy}.
\begin{conj}\label{cj11}
For all $a \ge 0$ and all $k \ge 0$ with $a+k \ge 1$ there holds
$$
|| 2^a \cdot 3^k || = 2a + 3k.
$$
\end{conj}

This conjecture is presented as a convenient form for summarizing  existing
knowledge; there is limited evidence for its truth, and it may well be false.
Indeed its truth would imply $\cpx{2^a} = 2a$, for all $a$.
Selfridge raised this special case in a contrary form,
asking the question whether there is some $a$ for which $\cpx{2^a} < 2a$
(see \cite{Guy}).

In this paper, we will investigate these questions by looking at numbers $n$ for
which the difference $\dft(n):=\cpx{n}-3\log_3 n$ is less than a given
threshold;
these sets we may call numbers
with integer complexity close to the lower bound.

\subsection{Main Results}

The fundamental issue making the complexity of an integer a complicated
quantity are: (1) It assumes the same value for many integers, because it
is logarithmically small; (2) It is hard to determine lower bounds for a given
value $\cpx{n}$, since
the dynamic programming tree is exponentially large. The feature (1) implies
there can be many tie values in
going down the tree, requiring a very large search, to determine any specific
complexity value.

We introduce a new invariant to study integer complexity.
\begin{defn}
The \emph{defect} of a natural number $n$ is given by
\begin{equation*}\label{defd}
\dft(n)=\cpx{n}-3\log_3 n
\end{equation*}
\end{defn}

The introduction of the defect simplifies things in that it provides
a more discriminating invariant: we show that $\dft(n) \ge 0$ and that
it separates integers into quite small equivalence
classes. In these equivalence classes powers of $3$ play a special role.
The following result establishes a conjecture of Arias de Reyna \cite[Conjecture 1]{Arias}.

\begin{thm}\label{power-of-3}
(1) For a given value $\delta$ of the defect, the set
$S(\delta) :=\{ m:~~\dft(m) = \delta\}$, is a
chain $\{ n\cdot 3^k: 0 \le k \le k(n)\}$ where $k(n)$ may be finite or
infinite.
The value $n$ is called the leader of the chain.

(2) The function $\dft( n \cdot 3^k)$ is non-increasing on the
sequence $\{ n \cdot 3^k : \, k\ge 0\}$.
This sequence has a finite number of leaders
culminating in a largest leader $n \cdot 3^L$, having the property that
$$
|| n \cdot 3^k|| = ||n \cdot 3^L|| + 3(k-L), ~~\mbox{for all}~~k \ge L.
$$
\end{thm}
\noindent 

The set of integers $n \cdot 3^k$ for $k \ge L$ are termed {\em
stable integers},
because their representation using $1$'s stabilizes into a predictable form
for $k \ge L$. This result is proved in Section \ref{sec21}.

The main results of the paper concern classifying integers having
small values of the defect. The
defect is compatible with the multiplication aspect of the dynamic
programming definition of the integer complexity, but it does not
fully respect the addition aspect.
The main method underlying the results of this paper is given in
Theorem~\ref{themethod}, which provides strong constraints on the dynamic
programming
recursion for classifying numbers of small defect. It allows construction of
sets of integers including all integers of defect below a specified bound $r$,
which may however include some additional integers. The method contains
adjustable parameters, and with additional work they sometimes permit exact
determination of these sets.

This main method has several applications. First, we use it to explictly
classify
all integers of defect below the bound $12 \dft(2)\approx 1.286$.
(Theorem ~\ref{computeresult}). This requires pruning the sets found using
Theorem ~\ref{themethod} to determine the sets below $k \dft(2)$ 
for $1\le k \le 12.$

Using this result we obtain an explicit classification of all integers having
defect at most $1$, as follows.

\begin{thm}
The numbers $n$ satisfying $0\le \dft(n)<1$ are precisely those that can be
written in one of the following forms, and have the following complexities:
\begin{enumerate}
\item $3^k$ for $k\ge 1$, of complexity $3k$
\item $2^a 3^k$ for $a\le 9$, of complexity $2a+3k$ (for $a$, $k$ not both
zero)
\item $5\cdot2^a 3^k$ for $a\le 3$, of complexity $5+2a+3k$
\item $7\cdot2^a 3^k$ for $a\le 2$, of complexity $6+2a+3k$
\item $19\cdot3^k$ of complexity $9+3k$
\item $13\cdot3^k$ of complexity $8+3k$
\item $(3^n+1)3^k$ of complexity $1+3n+3k$ (for $n\ne0$)
\end{enumerate}
Furthermore $n=1$ is the only number having defect exactly $1$.
\end{thm}

This result is established in Section \ref{sec61}.
Using a slightly more general result, which we present as Theorem \ref{computeresult},
one can obtain a generalization of Rawsthorne's results,
consisting of a description of all $E_r(k)$ for every finite $r \ge 0$, 
valid for all sufficiently large $k$, depending on $r$.
This answer also depends on the congruence class of $k \pmod{3}$. For
example, one has $E_2(3k) = \frac{64}{81} E(3k)$, 
$E_2(3k+1) =\frac{5}{6} E(3k+1)$ and $E_2(3k+2) = \frac{5}{6} E(3k+2)$, all
holding for $k \ge 4$.
For $E_5(k)$ all three residue classes have different formulas, valid for $k
\ge 5$.  This generalization will be described elsewhere (\cite{seq3}).

Secondly, the result can be used to obtain lower bounds on complexity
of certain integers, by showing they are excluded from sets containing all
integers of complexity at most $r$.
This  we use to prove Conjecture \ref{cj11} for $a \le 21$.

\begin{thm}\label{th11main}
For all $0\le a \le 21$ and any $k\ge 0$ having $a+k \ge 1$, there holds
$$
\cpx{2^a3^k}=2a+3k.
$$
\end{thm}

This result is established in Section \ref{sec62}.
It is possible to carry out computations establishing
the Conjecture \ref{cj11} for larger value of $a$,
as we shall describe in \cite{seq2}.

Thirdly, our main method can be used to estimate the magnitude of numbers below
$x$ having a given defect.

\begin{thm}
\label{indcount0}
For any $r >0$ the number of elements $A_r(x)$ smaller than $x$
which have complexity $\dft(n) <r$ satisfies
an upper bound, valid for all $x \ge 2$,
$$
A_r(x) \le C_r (\log x)^{\lfloor r \rfloor+1},
$$
where $C_r >0$ is an effectively computable constant depending on $r$.
\end{thm}

This result is proved in Section \ref{sec63}. It implies that the set of
possible defect values is unbounded.

\subsection{Discussion}

We first remark on computing $\cpx{n}$. The recursive definition permits
computing $\cpx{n}$ by dynamic programming, but it requires knowing
$\{ \cpx{k} : 1 \le k \le n-1\}$, so takes exponential time in the input size
of $n$ measured in bits. In particular, a straightforward
approach to computing $\cpx{x}$
requires on the order of $n^2$ steps. Srinivas and Shankar \cite{waset}
obtained an improvement on this, running in time $O(n^{\log_2 3})$.

We make some further remarks on Conjecture \ref{cj11}.
Let's specialize to $k=0$ and
consider an analogous question for prime powers,
concerning $\cpx{p^m}$ as $m$ varies.
It is clear that $\cpx{p^m} \le m \cdot \cpx{p}$, since we can
concatenate by multiplication $m$ copies of a good representation of $p$.
For which primes $p$ is it true that
$\cpx{p^m} = m \cpx{p}$ holds for all $m \ge 1$?
This is verified for $p=3$ by $\cpx{3^m} = 3m,$
and the truth of Conjecture \ref{cj11} requires that it hold
for $p=2$, with $\cpx{2^m} = 2m$.
However this  question has a negative answer for powers of $5$.
Here while $\cpx{5}=5$, one instead gets that $\cpx{5^6}=\cpx{15625}=29<6
\cdot\cpx{5}= 30$, as
\begin{eqnarray*}
15625 & = & 1+(1+1)(1+1)(1+1)(1+1+1)(1+1+1)\cdot \\
& & (1+(1+1)(1+1)(1+1)(1+1+1)(1+1+1)(1+1+1))
\end{eqnarray*}
This encodes the identity $5^5= 1 +72 \cdot 217$,
in which $72= 2^3 \cdot 3^2$ and $217= 1+ 2^3 \cdot 3^3$.
This counterexample for powers of $5$ leaves open the possibility
that there might exist a (possibly far larger) counterexample for powers of $2$,
that has not yet been detected. 

This discussion shows that Conjecture \ref{cj11}, if true,  
implies a kind of very strong arithmetic independence of powers of $2$ and
powers of $3$.  This would represent an important feature of the prime $2$
in integer complexity.  Conjecture \ref{cj11} has implications about the
number of nonzero digits in the expansion of $2^n$ in base $3$ as a function of
$n$; namely, if there existed a
large power of $2$ with a huge number of zero digits in its base $3$ expansion,
then this would give a (counter)-example achieving $\cpx{2^k}< 2k$.
Problems similar to this very special subproblem already appear difficult (see
Lagarias \cite{Lag09}).  A result of C.~L.~Stewart \cite{Stewart} yields a lower
bound on the number of nonzero digits appearing in the base $3$ expansion of
$2^n$, but it is tiny, being only $\Omega(\frac{\log n}{\log \log n})$.

The truth of $\cpx{2^n}= 2n$ would also immediately imply the lower bound 
$$
\limsup_{n\rightarrow\infty} \frac{\cpx{n}}{\log n}\ge \frac{2}{\log 2}.
$$
Computer experiments seem to agree with this prediction and even allow the
possibility of equality, see Iraids et al \cite{data2}.

There remain many interesting open questions concerning the classification of
integers given by the defect.  The first concerns the distribution of stable
and unstable integers. How many are there of each kind?  A second question
concerns the function  $M(n)$ that counts the number of distinct minimal
decompositions into $1$'s that a given integer $n$ has. How does this function
behave?  

Finally we remark that the set $\mathscr{D} := \{ \dft(n): n \ge 1 \}$ of all
defect values turns out to be a highly
structured set. In a sequel \cite{seq1}, we shall show that it is
a well-ordered set, of order type $\omega^\omega$, a fact related to some
earlier conjectures of Juan Arias de Reyna \cite{Arias}.

\section{Properties of the defect}
\label{secdft}

The defect is the fundamental tool in this paper; let us begin by
noting some of its basic properties.

\begin{prop}
\label{multdft}

(1) For all integers $a \ge 1$,
\[ \dft(a) \ge 0.\]
Here equality holds precisely for $a= 3^k$, $k \ge 1$.

(2) One has
\[ \dft(ab)\le \dft(a)+\dft(b),\]
and equality holds if and only if
$\cpx{ab}=\cpx{a}+\cpx{b}$.

(3) For $k\ge 1$,
\[\dft(3^k \cdot n) \le \dft(n)\]
and equality holds
if and only if $\cpx{3^k \cdot n}=3k+\cpx{n}$.

\end{prop}

\begin{proof}
(1) This follows from the result of Selfridge.  Since for $k\ge 1$,
$\cpx{3^k}=3k$, we have $\dft(3^k)=0$ for $k\ge 1$, while $\dft(1)=1$.  For the
converse, note that $3\log_3 n$ is only an integer if $n$ is a power of $3$.

(2) This is a direct consequence of the definition.

(3) This follows from (2), from noting that $\dft(3^k)=0$ for $k\ge 1$.
\end{proof}

Because $\cpx{3^k}=3k$ for $k\ge 1$, one might hope that in general,
$\cpx{3n}=3+\cpx{n}$ for $n>1$. However, this is not so; for instance,
$\cpx{107}=16$, but $\cpx{321}=18$.

The defect measures how far a given integer is from the upper bound $E(||n||)$,
given in terms of  the ratio $E(\cpx{n})/n$:

\begin{prop}
We have $\delta(1) =1$ and 
\label{dRformulae}
\begin{displaymath}
\delta(n)=\left\{ \begin{array}{ll}
3\log_3 \frac{E(\cpx{n})}{n}	& \mathrm{if}\quad \cpx{n}\equiv 0\pmod{3}, \\
3\log_3 \frac{E(\cpx{n})}{n} +2\,\delta(2)
	& \mathrm{if}\quad \cpx{n}\equiv 1\pmod{3},  \,\,  \mathrm{with} \; n >  1, \\
3\log_3 \frac{E(\cpx{n})}{n} +\delta(2)
	& \mathrm{if}\quad \cpx{n}\equiv 2\pmod{3}. 
\end{array} \right.
\end{displaymath}
In particular $E(\cpx{n})/n\ge 1$ for any $n \ge 1$.
\end{prop}

\begin{proof}
The proof is a straightforward computation using Selfridge's formulas for $E(k)$,
for $k = 3j+ i,$ $i=0,1,2$.
\end{proof}

\subsection{Stable Integers}\label{sec21}

This example above motivates the following definition.

\begin{defn}
A number $m$ is called \emph{stable} if $\cpx{3^k \cdot m}=3k+\cpx{m}$
holds for every $k \ge 1$.
Otherwise it is called \emph{unstable}.
\end{defn}

We have the following criterion for stability.

\begin{prop}
\label{stabdft}
The number $m$ is stable if and only if $\dft(3^k \cdot m)=\dft(m)$ for all $k\ge 0$.
\end{prop}

\begin{proof}
This is immediate from Proposition~\ref{multdft}(3).
\end{proof}

These results already suffice to prove the following
result, conjectured by Juan Arias de Reyna \cite{Arias}.

\begin{thm}
\label{cj1}
(1) For any $m \ge 1$, there exists a finite $K\ge 0$ such that
$3^K m$ is stable.

(2) If the defect $\dft(m)$ satisfies $0 \le \dft(m)<1$, then $m$ itself is
stable.
\end{thm}

\begin{proof}[Proof of Theorem~\ref{cj1}]
(1) From Proposition~\ref{multdft}, we have that for any $n$,
$\dft(3n)\le
\dft(n)$, with equality if and only if $\cpx{3n}=\cpx{n}+3$.  More generally,
$\dft(3n)=\dft(n)-(\cpx{n}+3-\cpx{3n})$, and so the difference
$\dft(n)-\dft(3n)$ is always an integer.
This means that the sequence
$\dft(m), \dft(3m), \dft(9m), \ldots$ is non-increasing, nonnegative, and
can only decrease in integral amounts;
hence it must eventually stabilize. Applying Proposition~\ref{stabdft} proves
the theorem.

(2) If $\dft(m)<1$, since all $\dft(n) \ge 0$ there is no room to remove
any integral
amount, so $m$ must be stable.
\end{proof}

Note that while this proof shows that for any $n$ there exists $K$ such that
$3^K n$ is stable, it yields no upper bound on such a $K$.  We will give a more
constructive proof and show how to compute such a $K$ in \cite{seq2}.

The value of the defect separates the integers into small classes, whose members
differ only by powers of $3$.

\begin{prop}
\label{eqdefect}
Suppose that $m$ and $n$ are two positive integers, with $m>n$.

(1) If $q:= \dft(n)- \dft(m)$ is rational, then it is necessarily a nonnegative integer,
and furthermore $m=n \cdot 3^k$ for some $k \ge 1$.

(2) If $\dft(n) = \dft(m)$ then $m= n \cdot 3^k$ for some $k \ge 1$ and furthermore
\[ || n \cdot 3^j|| = 3j+ || n ||\qquad\mathrm{for}\ 0 \le j \le k.\]
In particular $\dft(n)= \dft(m)$ implies $\cpx{n}\equiv \cpx{m} \pmod{3}$.

\end{prop}

\begin{proof}
(1) If $q=\dft(n)-\dft(m)$ is rational, then $k=\log_3(m/n)\in \Q$ is
rational;
since $m/n$ is rational, the only way this can occur is if $\log_3(m/n)$ is an
integer $k$, in which case, since $m > n,$ $m = n \cdot 3^k$ with $k \ge 1$.
It then follows from the definition of defect that
$q=\cpx{n}+3k-\cpx{m}$.

(2) By (1) we know that $m=n \cdot 3^k$ for some $k \ge 1$. By
Proposition \ref{multdft} (3)
we have $\dft(n \cdot 3^j) \le \delta (n)$,  for  $j \ge 0$ 
and it also gives
$\dft(m)= \dft(n \cdot 3^k) \le \dft(n \cdot 3^j),$ for $0 \le j \le k$.
Since $\dft(m)=\dft(n)$ by hypothesis, this gives $\dft(n \cdot 3^j) =
\dft(n)$,
so that $ ||n \cdot 3^j|| = 3j+ ||n||: 0 \le j \le k.$
\end{proof}

The results so far suffice to prove Theorem ~\ref{power-of-3}.

\begin{proof}[Proof of Theorem ~\ref{power-of-3}]
(1) This follows from Proposition \ref{eqdefect}(2).

(2) The non-increasing assertion  follows from Proposition \ref{multdft}(3).
The finiteness of the number of leaders in a sequence $3^k \cdot n$ follows  from Theorem \ref{cj1} (1).
\end{proof}

\subsection{Leaders}\label{sec22}

Again because $\cpx{3n}$ is not always equal to $3+\cpx{n}$, it makes sense to
introduce the following definition:

\begin{defn}
We call a natural number $n$ a \emph{leader} if it cannot be written
most-efficiently as $3m$ for some $m$; i.e., if either $3\nmid n$, or, if $3\mid n$, then $\cpx{n}<3+\cpx{n/3}$.
\end{defn}

For example, $107$ is a leader since $3\nmid 107$, and $321$ is also a leader
since $\cpx{321}=18<3+16=3+\cpx{107}$. However, $963$ is not a leader, as
$\cpx{963}=21=3+\cpx{321}$.
Leaders can be stable or unstable. In this example $107$ is unstable, but by
Theorem \ref{cj1}
some multiple $3^K \cdot 107$ will be stable, and the smallest such multiple
will be a stable leader.

We have the following alternate characterization of leaders:
\begin{prop}
\label{1stofdft}
(1) A number $n$ is a leader if and only if it is the smallest number having its given defect value.

(2) For any natural number $m$, there is a unique leader $n\le m$ such that
$\dft(n)= \dft(m)$. For it $m=n \cdot 3^k$ for some $k \ge 0$.
\end{prop}

\begin{proof}
(1) If this were false, there would a leader $n$ with  some $n' < n$ with $\delta(n')=\delta(n)$.
By Proposition \ref{eqdefect} (2) $n = 3^k \cdot n'$ with $k \ge 1$ and 
$||n' \cdot 3^j|| = 3j + ||n'||$ for $0 \le j \le k$. But then $n/3 = n' \cdot
3^{k-1}$ is an integer
and $||n/3|| = ||n'||+ 3k -3= ||n||-3$, which contradicts $n$ being a leader.

Conversely, if $n$ is the first number of its defect and is divisible by $3$,
then we cannot have $\cpx{n}=\cpx{n/3}+3$, or else by Proposition~\ref{multdft}
we would obtain $\dft(n)=\dft(n/3)$, contradicting minimality.

(2) Pick $n$ to be the smallest number such that $\dft(n)=\dft(m)$; this is
the unique leader satisfying $\dft(n)=\dft(m)$. Then $m=3^k n$ for some
$k\ge 0$ by Proposition~\ref{eqdefect}.
\end{proof}

To summarize, if $\delta$ occurs as a defect, then the set of integers
$$
N(\delta) := \{m:\, \dft(m)= \delta\},
$$
having a given defect value $\delta$
has a smallest element that is a leader. If this leader $n$ is unstable, then
$N(\delta) =\{ 3^j \cdot n: 0 \le j \le j(\delta)\}$. If this leader
is stable, then $N(\delta)= \{ 3^j \cdot n: \, j \ge 0\}$ is an infinite set.
Furthermore if $3 \nmid n$ then $n$ is a leader, and there is a unique $K= K(n)
\ge 0$
such that $n' = 3^K n$ is a stable leader.

\section{Good factorizations and solid numbers}

Given a natural number $n>1$, by the dynamic programming definition of
complexity there are either two numbers $u$ and $v$, both smaller than $n$,
such that $n=u\cdot v$ and $\cpx{n}=\cpx{u}+\cpx{v}$, or such that $n=u+v$ and
$\cpx{n}=\cpx{u}+\cpx{v}$. In the case $u$ and $v$
such that $n=u+v$, and $\cpx{n}=\cpx{u}+\cpx{v}$ we say
$n$ is {\em additively reducible}. In the case $n=u\cdot v$ and
$\cpx{n}=\cpx{u}+\cpx{v}$ we say
$n$ is {\em multiplicatively reducible}.
Some numbers $n$ are reducible in both senses. For instance, $10=9+1$ with
$\cpx{10}=\cpx{9}+\cpx{1}$, and $\cpx{10}=2\cdot 5$ with
$\cpx{10}=\cpx{2}+\cpx{5}$.

\subsection{Additive Irreducibility and Solid Numbers}

We introduce terminology for numbers
not being additively reducible.

\begin{defn}
We will say that a natural number $n$ is {\em additively irreducible} if it
cannot be written most-efficiently as a sum, i.e., for all $u$ and
$v$ such that $n=u+v$, we have $\cpx{n}<\cpx{u}+\cpx{v}$. We call such values of
$n$ {\em solid numbers}.
\end{defn}

The first few solid numbers are
\begin{align*}
\{1, 6, 8, 9, 12, 14, 15, 16, 18, 20, 21, 24, 26, 27, \ldots\}
\end{align*}

It can be shown that $3^n$ is a solid number for $n\ge 2$, and so there are
infinitely many solid numbers.  Experimental evidence suggests that a positive
fraction of integers below $x$ are solid numbers, as $x \to \infty$.

\subsection{Multiplicative Irreducibility and Good Factorizations}

We introduce further terminology for factorizations that respect complexity.

\begin{defn}
A factorization $n=u_1\cdot u_2\cdots u_k$ is a \emph{good factorization} of $n$
if $n$ can be written most-efficiently as $u_1\cdot u_2\cdots u_k$, i.e., if the
following equality holds:
\begin{displaymath}
\cpx{n}=\cpx{u_1}+\cpx{u_2}+\ldots+\cpx{u_k}.
\end{displaymath}
The factorization containing only one factor is automatically good; this will be
called a \emph{trivial good factorization}.
\end{defn}

\begin{prop}
\label{goodfac}
If $n=n_1\cdot n_2\cdot \ldots \cdot n_k$ is a good factorization then for any
nonempty subset $I\subset \{1,2,\dots,k\}$ the product $m=\prod_{j\in I} n_j$
is a good factorization of $m$.
\end{prop}

\begin{proof}
If the factorization of $m$ were not good, then we would have
\begin{displaymath}
\cpx{m}<\sum_{j\in I} \cpx{n_j}
\end{displaymath}
But then
\begin{displaymath}
\cpx{n} = \Bigl \Vert m \prod_{j\notin I} n_j\Bigr \Vert
<\sum_{j\in I} \cpx{n_j}+\sum_{j\notin I} \cpx{n_j}
=\sum_{j=1}^k \cpx{n_j}
\end{displaymath}
and the given factorization of $n$ would not be a good factorization.
\end{proof}

\begin{prop}
\label{goodfacconcat}
(1) If $n=n_1\cdot n_2 \cdot ... \cdot n_k$ is a good factorization, and each
$n_i=n_{i,1} \cdot
\ldots \cdot n_{i,l_i}$ is a good factorizations, then so is $n=\prod_{i=1}^k
\prod_{j=1}^{l_i} n_{i,j}$.

(2) If $n=n_1\cdot n_2\cdot \ldots \cdot n_k$ is a good factorization, and $I_1,
I_2, \ldots, I_l$ is a partition of $\{1,\ldots,k\}$, then letting
$m_i=\prod_{j\in I_i} n_j$, we have that $n=\prod_{i=1}^l m_i$ is a good
factorization.
\end{prop}

\begin{proof}
(1) We have that $\cpx{n_i}=\sum_{j=1}^{l_i} \cpx{n_{i,j}}$ and
$ \cpx{n}=\sum_{i=1}^k \cpx{n_i}$, so
\[\cpx{n}=\sum_{i=1}^k \sum_{j=1}^{l_i}
\cpx{n_{i,j}}
\]
and we are done.

(2) This follows from Proposition~\ref{goodfac} together with (1).
\end{proof}

\begin{defn}
We will say that
a natural number $n$ is {\em multiplicatively irreducible}
(abbreviated \emph{$m$-irreducible}) if $n$ has no nontrivial good
factorizations.
\end{defn}

Proposition~\ref{goodfacconcat}(2) shows $n$ is $m$-irreducible if and only if
all nontrivial factorizations $n=uv$ have $\cpx{n}<\cpx{u}+\cpx{v}$.
Thus a prime number $p$ is automatically $m$-irreducible since the only
factorization is $p=p\cdot1$ and obviously we have
$\cpx{p}<\cpx{p}+1=\cpx{p}+\cpx{1}$. However, the converse does not hold.
For instance, $46$ is a composite number which is $m$-irreducible.

\begin{prop}
\label{facexists}
Any natural number has a good factorization into $m$-irreducibles.
\end{prop}

\begin{proof}
We may apply induction and assume that any $m<n$ has a factorization into
$m$-irreducibles. If $n$ is $m$-irreducible, we are done. Otherwise, $n$ has a
good factorization $n=uv$. Observe that $n=n\cdot 1$ is never a good
factorization, since $\cpx{1}=1$; hence, $u$, $v<n$. Then the induction
hypothesis implies that $u$ and $v$ have good factorizations into
$m$-irreducibles.  Multiplying these factorizations together and applying
Proposition~\ref{goodfacconcat}, we obtain a good factorization of $n$ into
$m$-irreducibles.
\end{proof}

Good factorizations into $m$-irreducibles need not be unique. For
$4838 = 2 \cdot 41 \cdot 59$,
we find that $2\cdot(41\cdot59)$, $(2\cdot59)\cdot41$ and $(2\cdot41)\cdot59$
are all good factorizations, but the full factorization $2\cdot41\cdot59$ is
not a good factorization. (Thanks to Juan Arias de Reyna for this example.) This is deducible from the following data:
\begin{gather*}
\cpx{2\cdot41\cdot 59}=27,\\
\cpx{2}=2,\quad
\cpx{41}=12,\quad
\cpx{59}=14.\\
\cpx{2\cdot41}=13,\quad
\cpx{2\cdot59}=15,\quad
\cpx{41\cdot59}=25,
\end{gather*}

\subsection{Good factorizations and leaders}

The next two propositions show how the notion of good factorization interacts
with leaders and stability.

\begin{prop}
\label{leaderfac}
Let $n=n_1\cdot n_2\cdots n_r$ be a good factorization. If $n$ is a leader
then each of the factors $n_j$ is a leader.
\end{prop}

\begin{proof}
Suppose otherwise; without loss of generality, we may assume that $n_1$ is not a
leader, so $3\mid n_1$ and $\cpx{n_1}=3+\cpx{n_1/3}$. So $3\mid n$ and
\begin{multline*}
\cpx{n/3}=\cpx{(n_1/3)\cdot n_2\cdot\ldots\cdot n_r}\le
\cpx{n_1/3}+\sum_{j=2}^r \cpx{n_j}\\ =
\cpx{n_1}-3+\sum_{j=2}^r \cpx{n_j}=\cpx{n}-3.
\end{multline*}
Since $\cpx{n}\le 3+\cpx{n/3}$, we have $\Vert n\Vert= 3+\Vert n/3\Vert$,
and thus $n$ is not a leader.
\end{proof}

\begin{prop}
\label{stablefac}
Let $n=n_1\cdot n_2\cdots n_r$ be a good factorization. If $n$ is stable,
then each of its factors $n_j$ is stable.
\end{prop}

\begin{proof}
Suppose otherwise.  Without loss of generality, we may assume that $n_1$ is
unstable; say $\cpx{3^k n_1}<\cpx{n_1}+3k$.  So
\begin{multline*}
\cpx{3^k n}=\cpx{(3^k n_1)\cdot n_2\cdot\ldots\cdot n_r}\le
\cpx{3^k n_1}+\sum_{j=2}^r \cpx{n_j}\\<
\cpx{n_1}+3k+\sum_{j=2}^r \cpx{n_j}=\cpx{n}+3k.
\end{multline*}
and thus $n$ is not stable.
\end{proof}

Assembling all these results we deduce that being a leader and being stable are
both inherited properties for subfactorizations of good factorizations.

\begin{prop}
\label{lsfac2}
Let $n=n_1\cdot n_2\cdots n_r$ be a good factorization, and $I$ be a nonempty
subset of $\{1,\ldots,r\}$; let $m=\prod_{i\in I} n_i$.  If $n$ is a leader,
then so is $m$.  If $n$ is stable, then so is $m$.
\end{prop}

\begin{proof}
Immediate from Proposition~\ref{leaderfac}, Proposition~\ref{stablefac}, and
Proposition~\ref{goodfacconcat}.(2).
\end{proof}

\section{The Classification Method}
\label{mainlem}

Here, we state and prove a result (Theorem \ref{themethod}) that will be our
primary tool for the rest of the paper. By applying it repeatedly, for any
$r>0$, we can put restrictions on what integers $n$ can satisfy $\dft(n)<r$.

\begin{defn}
(1) For any real $r\ge0$, define $A_r$ to be $\{n\in\mathbb{N}:\dft(n)<r\}$.

(2) Define
$B_r$ to be the set consisting of those elements of $A_r$ that are leaders.
\end{defn}

While $A_r$ is our main object of interest, it turns out to be easier
and more natural to deal with $B_r$.
Note that knowing $B_r$ is enough to determine $A_r$, as
expressed in the following proposition:

\begin{prop}
\[A_r=\{ 3^k n: n\in B_r, k\ge 0 \}\]
\end{prop}

\begin{proof}
If $n\in B_r$, then $\dft(3^k n)\le \dft(n)<r$, so $3^k n\in A_r$.
Conversely, if $m\in A_r$, by Proposition~\ref{1stofdft}(2) we can take $n\ge 1$
and $k\ge 0$ such that $n$ is a leader, $m=3^k n$, and $\dft(m)=\dft(n)$;
then $n\in B_r$ and we are done.
\end{proof}

We now let $\alpha >0$ be a real parameter, specifiable in
advance. The main result puts constraints on the allowable forms of
the dynamic programming recursion (most efficient representations) to compute
integers in $B_{(k+1) \alpha}$
in terms of integers in $B_{j \alpha}$ for $1 \le j \le k$.
However there are some exceptional cases that must be considered separately in
the theorem;
fortunately, for any $\alpha<1$, there are only finitely many.  We will
collect these into a set we call $T_\alpha$.
\begin{defn}
Define $T_\alpha$ to consist of $1$ together with those $m$-irreducible
numbers $n$ which satisfy
\[\frac{1}{n-1}>3^{\frac{1-\alpha}{3}}-1\]
and do not satisfy $\cpx{n}=\cpx{n-b}+\cpx{b}$ for any solid numbers $b$ with
$1<b\le n/2$.
\end{defn}

Observe that for $0< \alpha<1$, the above inequality is equivalent to
\[n<(3^{\frac{1-\alpha}{3}}-1)^{-1}+1\] and hence $T_\alpha$ is a finite set.
For $\alpha \ge 1$, the inequality is trivially satisfied and so
$T_\alpha= T_{1}$. We do not know whether $T_1$ is a finite or an infinite set.
However in our computations we will always choose values $0< \alpha <1.$

We can now state the main classification result, which puts strong
constraints on the form of most efficient decompositions on numbers in sets
$B_{(k+1)\alpha}$.

\begin{thm}
\label{themethod}
Suppose  $0< \alpha <1$ and that $k\ge1$.
Then any $n\in B_{(k+1)\alpha}$ can be most-efficiently
represented in (at least) one of the following forms:
\begin{enumerate}
\item
For $k=1$,
there is either a good factorization $n=u\cdot v$ where
$u,v\in {B}_\alpha$, or a good factorization $n=u\cdot v\cdot w$ with
$u,v,w\in {B}_\alpha$; \\
For $k \ge 2$, there is a good factorization $n=u \cdot v$ where $u\in
B_{i\alpha}$,
$v\in B_{j\alpha}$ with $i+j=k+2$ and $2\le i, j\le k$.
\item $n=a+b$ with $\cpx{n}=\cpx{a}+\cpx{b}$, $a\in A_{k\alpha}$, $b\le a$ a
solid number and
\[\dft(a)+\cpx{b}<(k+1)\alpha+3\log_3 2.\]
\item There is a good factorization $n=(a+b)v$ with $v\in B_\alpha$ and $a$
and $b$ satisfying the conditions in the case (2) above.
\item $n\in T_\alpha$ (and thus in particular either $n=1$ or
$\cpx{n}=\cpx{n-1}+1$.)
\item There is a good factorization $n = u\cdot v$ with $u\in T_\alpha$ and
$v\in B_\alpha$.
\end{enumerate}
\end{thm}

We will prove Theorem ~\ref{themethod} in Section \ref{sec43},
after establishing a preliminary combinatorial lemma in Section \ref{sec42}.

To apply Theorem \ref{themethod}, one recursively constructs
from given sets $B_{j\alpha}^{*}$, $A_{j\alpha}^{*}$ for $1 \le j \le k-1$
which contain $B_{j\alpha}, A_{j\alpha}$, respectively,
the set of all $n$ satisfying the relaxed conditions (1)-(5) obtained replacing
$B_{j\alpha}$ by
$B_{j\alpha}^{\ast}$ and $A_{j \alpha}$ by $A_{j \alpha}^{\ast}$.
This new set $B_{(k+1)\alpha}^{\ast\ast}$ contains the set $B_{(k+1)\alpha}$
we want.
Sometimes we can, by other methods, prune some elements from
$B_{(k+1)\alpha}^{\ast\ast}$ that do not belong
to $B_{(k+1)\alpha}$, to obtain a new approximation $B_{(k+1)\alpha}^{\ast}$.
This then determines $A_{(k+1)\alpha}^{\ast} := \{ 3^k n: \, k \ge 0, n \in
B_{(k+1)\alpha}^{\ast}\}$,
permitting continuation to the next level $k+2$.
We will present two applications of this construction:
\begin{enumerate}
\item
To get an upper bound on the cardinality of $B_{(k+1)\alpha}$ of numbers below
a given bound $x$.
\item
To get a lower bound for the complexity $\cpx{n}$ of a number $n$ by showing it
does not belong to a given set $B_{k \alpha}^{*}$; this excludes it from $B_{k
\alpha}$, whence $\cpx{n} \ge 3 \log_3 n + k \alpha$.
\end{enumerate}
In some circumstances we can obtain the exact sets $B_{k \alpha}$ and $A_{k
\alpha}$ for $1 \le k \le k_0$, i.e. we recursively construct $B_{k
\alpha}^{\ast}$ so that $B_{k\alpha}^{\ast} = B_{k \alpha}$.
This requires a perfect pruning operation at each step. Here a good choice of
the parameter $\alpha$ is helpful.

In applications we will typically not use the full strength of Theorem
\ref{themethod}.  Though the representations it yields are most efficient, the
proofs  will typically not
use this fact.  Also, in the addition case (2), the requirement that
$\dft(a)+\cpx{b}<(k+1)\alpha+3\log_3 2$ implies the weaker
requirement that just $\cpx{b}<(k+1)\alpha+3\log_3 2$.
The latter relaxed condition is easier to check, but it does enlarge the initial
set $B_{(k+1) \alpha}^{\ast\ast}$ to be pruned.

\subsection{A Combinatorial Lemma}\label{sec42}

We establish a combinatorial lemma regarding decomposing a sum of
real numbers into blocks.
\begin{lem}
\label{blocklem}
Let $x_1, x_2, \dots, x_r>0$ be real numbers such that $\sum_{i=1}^r x_i <
k+1$, where $k\ge1$ is a natural number.

(1) If $k \ge 2$ then either there is some $i$ with
$x_i\ge k$, or else we may find a partition $A\cup B$ of the set
$\{1,2,\dots, r\}$ such that
\[
\sum_{i\in A}x_i<k,\qquad\sum_{i\in B}x_i<k.
\]

(2) If $k=1$ then either there is some $i$ with $x_i\ge 1$, or else we may find
a partition
$A\cup B\cup C$ of the set $\{1,2,\dots, r\}$ such that
\[
\sum_{i\in A}x_i<1,\qquad\sum_{i\in B}x_i<1,\qquad\sum_{i\in C}x_i<1.
\]
\end{lem}

\begin{proof}
(1) Suppose $k \ge 2$.
Let us abbreviate $\sum_{i\in S}x_i$ by $\sum S$.  Among all partitions $A\cup
B$ of $\{1,\ldots,r\}$, take one that minimizes $|\sum A - \sum B|$, with $\sum
A\ge\sum B$.  Suppose that $\sum A\ge k$; then since $\sum A+\sum B<k+1$, we
have $\sum B<1$, and so $\sum A-\sum B>k-1$. So pick $x_i\in A$ and let
$A'=A\setminus\{i\}$, $B'=B\cup\{i\}$.  If $\sum A'>\sum B'$, then 
$|\sum A'-\sum B'|=\sum A-\sum B - 2x_i<\sum A-\sum B$, contradicting
minimality, so $\sum A'\le\sum B'$.  
So $\sum B'-\sum A'\ge \sum A-\sum B$, i.e., 
\[
x_i\ge \sum A-\sum B>k-1. 
\] 
Now $i$ was an arbitrary element of $A$; this means that $A$ can
have at most one element, since otherwise, if $j\ne i\in A$, we would have $\sum
A\ge x_i + x_j$ and hence $x_j\le \sum A-x_i \le \sum B<1$, but also $x_j>k-1$,
contradicting $k\ge 2$.  Thus $A=\{i\}$ and so $x_i\ge k$.

(2) Here $k=1$. Assume that $x_1\ge x_2\ge\cdots\ge x_r$. If $x_1\ge1$ we are
done.  Otherwise,
if $r\le3$, we can partition $\{1,\ldots,r\}$ into singletons.

For $r\ge 4$, assume by induction the lemma is true for all sets of numbers with
strictly less than $r$ elements.  Let $y=x_{r-1}+x_r$. We must have $y<1$
because otherwise $x_{r-3}+x_{r-2}\ge x_{r-1}+x_r\ge 1$ and we get $\sum_{i=1}^r
x_i\ge 2$ in contradiction to the hypothesis.
Hence, if we define $x'_1=x_1$, \dots, $x'_{r-2}=x_{r-2}$, $x'_{r-1}=y$, we have 
$\sum_{i=1}^{r-1}x'_i=\sum_{i=1}^r x_i<2$, and $x'_i<1$ for all $i$.  By the
inductive hypothesis, then, there exists a paritition 
$A'\cup B'\cup C'=\{1,\dots, r-1\}$ with
 \[
  \sum_{i\in A'}x'_i<1,\quad
\sum_{i\in B'}x'_i<1,\quad\sum_{i\in C'}x'_i<1. 
\]
Replacing $x'_{r-1}$ with $x_{r-1}$ and $x_r$, we get the required partition of
$\{1,\ldots,r\}$.
\end{proof}

For $k=1$ the example taking $\{ x_1, x_2, x_3\} = \{ 3/5, 3/5, 3/5\}$ shows
that a partition into three sets is sometimes necessary.

\subsection{Proof of the Classification Method} \label{sec43}

\begin{proof}[Proof of Theorem \ref{themethod}.]
Suppose $n\in B_{(k+1)\alpha}$; take a most-efficient representation of $n$,
which is either $ab$, $a+b$, or $1$.  If $n=1$, then $n\in T_\alpha$ and we are
in case (4).  So suppose $n>1$.

If $n$ is $m$-irreducible, we will pick a way of writing $n=a+b$ with
$\cpx{n}=\cpx{a}+\cpx{b}$, $a\ge b$, and $b$ is solid.  There is necessarily a
way to do this, since one way to do so is to write $n=a+b$ with
$\cpx{n}=\cpx{a}+\cpx{b}$ and $b$ minimal.  Since this is possible, then, if
there is a way to choose $a$ and $b$ to have $b>1$, do so; otherwise, we must
pick $b=1$.  In either case,
\[\cpx{a}+\cpx{b}=\cpx{n}<3\log_3(a+b)+(k+1)\alpha\le3\log_3(2a)+(k+1)\alpha,\]
so $\dft(a)+\cpx{b}<(k+1)\alpha+3\log_3 2$.

If $a\in A_{k\alpha}$, we are in case (2).  Otherwise, we have
\begin{eqnarray*}
3\log_3 a+k\alpha+\cpx{b}\le \cpx{a}+\cpx{b}=\cpx{n}<\\
3\log_3(a+b)+(k+1)\alpha \le 3\log_3(2a)+(k+1)\alpha,
\end{eqnarray*}
so $\cpx{b}<3\log_3 2 +\alpha$; since $\alpha<1$, we have $\cpx{b}\le2$ and
thus $b\le2$.  Because $b$ is solid, we have $b=1$.  By assumption, we only
picked $b=1$ if this choice was forced upon us, so in this case, we must have
that $n$ does not satisfy $\cpx{n}=\cpx{n-b}+\cpx{b}$ for any solid $b$ with
$1<b\le n/2$.

Since $b=\cpx{b}=1$ we have
$3\log_3 a+k\alpha+1<3\log_3(a+1)+(k+1)\alpha$; since $\alpha<1$, solving for
$a$, we find that
\[
\frac{1}{n-1}=\frac{1}{a}>3^{\frac{1-\alpha}{3}}-1.
\]
Thus, $n\in T_\alpha$ and we are in case (4).

Now we consider the case when $n$ is not $m$-irreducible.  Choose a good
factorization of $n$ into $m$-irreducible numbers, $n=\prod_{i=1}^r m_i$; since
$n$ is not $m$-irreducible, we have $r \ge 2$.  Then we have $\sum_{i=1}^r
\dft(m_i)=\dft(n)<(k+1)\alpha$.  Note that since we assumed $n$ is a leader,
every product of a nonempty subset of the $m_i$ is also a leader by
Proposition~\ref{lsfac2}.  We now have two cases.

{\em Case 1.} $k\ge2$.

Now by  Lemma~\ref{blocklem}(1), either there
exists an $i$ with $\dft(m_i)\ge k\alpha$, or else we can partition the
$\dft(m_i)$ into two sets each with sum less than $k\alpha$.

In the latter case, we may also assume these sets are nonempty, as if one is
empty, this implies that $\dft(n)<k\alpha$, and hence any partition of the
$\dft(m_i)$ will work; since $r\ge 2$, we can take both these sets to be
nonempty.  In this case, call the products of these two sets $u$ and $v$, so
that $n=uv$ is a good factorization of $n$.  Then
$\dft(u)+\dft(v)<(k+1)\alpha$, so if we let $(i-1)\alpha$ be the largest
integral multiple of $\alpha$ which is at most $\dft(u)$, then letting
$j=k+2-i$, we have $\dft(v)<j\alpha$.  So $i+j=k+2$; furthermore, since
$i\alpha$ is the smallest integral multiple of $\alpha$ which is greater than
$\delta(u)$, and $\delta(u)<k\alpha$, we have $i\le k$, so $j\ge
2$.  If also $i\ge 2$ then $j\le k$, and so we are in case (1).  If instead
$i=1$, then we have $u\in B_\alpha \subseteq B_{2\alpha}$, and $v\in
B_{k\alpha}$ (since $\delta(v)<k\alpha$), so we are again in case
(1) if we take $i=2$ and $j=k$.

If such a partition is not possible, then let $u$ be an $m_i$ with
$\dft(m_i)\ge k\alpha$, and let $v$ be the product of the other $m_i$, so that
once again $n=uv$ is a good factorization of $n$.  Since
$\dft(u)+\dft(v)=\dft(n)$, we have $\dft(v)<\alpha$, and so $v\in
B_\alpha$.  Finally, since $u$ is $m$-irreducible and an element of
$B_{(k+1)\alpha}$, it satisfies the conditions of either case (2) or case (4),
and so $n$ satisfies the conditions of either case (3) or case (5).

{\em Case 2.} $k=1$.

Now by Lemma~\ref{blocklem}(2), either there
exists an $i$ with $\dft(m_i)\ge \alpha$, or else we can partition the
$\dft(m_i)$ into  three sets each with sum less than $\alpha$.

In the latter case, we may also assume at least two of these sets are nonempty,
as otherwise $\dft(n)<\alpha$, and hence any
partition of the $\dft(m_i)$ will work.
If there are two nonempty sets, call the products of these two sets $u$ and
$v$, so that $n=uv$ is a good factorization of $n$.  If there are three
nonempty sets, call their products $u, v, w$, so that
$n=uvw$ is a good factorization of $n$.  Thus we are in case (1) for $k=1$.

If such a partition is not possible, then we repeat the argument in Case 1
above, determining that $n$ satisfies one of the conditions of cases (3) or (5).
\end{proof}

\section{Determination of all elements of defect below a given bound $r$}
\label{sec5}

In this section we determine all elements of $A_{r}$ for certain small $r$,
using Theorem ~\ref{themethod}
together with a pruning operation.

\subsection{Classification of numbers of small defect}\label{sec50}

We will now choose as our parameter
\[
\alpha := \dft(2) = 2 - 3 \log_{3} 2 \approx 0.107.
\]
The choice of this parameter is motivated by Theorem \ref{thm-delta2} below.
We use above method to
inductively compute $A_{k\dft(2)}$
and $B_{k\dft(2)}$ for $0\le k \le 12$.  Numerically,
$1.286<12\dft(2)<1.287$.
The following result classifies all integers in $A_{12 \dft(2)}$.

\begin{thm}
\label{computeresult} {\em (Classification Theorem)}
The numbers $n$ satisfying $\dft(n)<12\dft(2)$ are precisely those that can
be written in at least one of the following forms, which have the indicated
complexities:
\begin{enumerate}
\item $3^k$ of complexity $3k$ (for $k\ge 1$)
\item $2^a 3^k$ for $a\le11$, of complexity $2a+3k$ (for $a$, $k$ not both zero)
\item $5\cdot 2^a 3^k$ for $a\le6$, of complexity $5+2a+3k$
\item $7\cdot 2^a 3^k$ for $a\le5$, of complexity $6+2a+3k$
\item $19\cdot 2^a 3^k$ for $a\le 3$, of complexity $9+2a+3k$
\item $13\cdot 2^a 3^k$ for $a\le 2$, of complexity $8+2a+3k$
\item $2^a(2^b3^l+1)3^k$ for $a+b\le 2$, of complexity $2(a+b)+3(l+k)+1$ (for
$b$, $l$ not both zero).
\item $1$, of complexity $1$
\item $55\cdot 2^a 3^k$ for $a\le 2$, of complexity $12+2a+3k$
\item $37\cdot 2^a 3^k$ for $a\le 1$, of complexity $11+2a+3k$
\item $25\cdot3^k$ of complexity $10+3k$
\item $17\cdot3^k$ of complexity $9+3k$
\item $73\cdot3^k$ of complexity $13+3k$
\end{enumerate}

In particular, all numbers $n>1$ with $\dft(n)<12\dft(2)$ are stable.
\end{thm}

This list is redundant; for example list (7) with $a=0, b=1, l=1$ gives $7
\cdot 3^k$, which overlaps list (4) with $a=0$.
But the given form is convenient for later purposes.
In the next section we will give several applications
of this result. They can be derived knowing only the statement of
this theorem, without its proof, though one will also require
Theorem~\ref{themethod}.

The detailed proof of this theorem is given in the rest of this section.
The proof recursively determines all the sets $A_{k \dft(2)}$ and $B_{k
\dft(2)}$ for $1 \le k \le 12.$
It is possible to extend this method to values $k \dft(2)$
with $k > 12$ but it is tedious.
In a sequel paper \cite{seq2}, we will present a method for automating these
computations.

\subsection{Base case}

The use of $\dft(2)$ may initially seem like an odd choice of
step size. Its significance is shown by the following base case, which
is proved using Rawsthorne's result that $E_1(k)\le(8/9)E(k)$ 
(with equality for $k \ge 8$).

\begin{thm}\label{thm-delta2}
If $\dft(n)\ne 0$, then $\dft(n)\ge \dft(2)$.  Equivalently, if $n$ is not
a power of $3$, then $\dft(n)\ge \dft(2)$.
\end{thm}

\begin{proof}
We apply Proposition \ref{dRformulae}. There are four cases.

Case 1. If $n=1$, then $\dft(n)=1\ge \dft(2)$.

Case 2. If $\cpx{n}\equiv 2 \pmod{3}$, then
$$\dft(n)=\dft(2)+3\log_3 \frac{E(\cpx{n})}{n}\ge \dft(2).$$

Case 3. If $\cpx{n}\equiv 1 \pmod{3}$ and $n> 1$, then
$$\dft(n)=2\dft(2)+3\log_3 \frac{E(\cpx{n})}{n}\ge 2\dft(2)\ge \dft(2).$$

Case 4. If $\cpx{n}\equiv 0\pmod{3}$, then $\dft(n)=3\log_3 (E(\cpx{n})/n)$.  We
know that in this case $n=E(\cpx{n})$ if and only if $n$ is a power of $3$ if
and only if $\dft(n)=0$.  So if $\dft(n)\ne 0$, then $n\le E_1(\cpx{n})$. But
$E_1(\cpx{n})\le (8/9)E(\cpx{n})$, so $E(\cpx{n})/n\ge 9/8$, so $\dft(n)\ge
3\log_3 \frac{9}{8} = 3\dft(2)\ge \dft(2)$.
\end{proof}

The proof above also establishes:

\begin{prop}\label{basecase}
$B_0=\emptyset$, and $B_{\dft(2)}=\{3\}$.
\end{prop}

To prove Theorem \ref{computeresult} 
we  will use Theorem \ref{themethod} for the ``inductive step''. However, while
Theorem \ref{themethod}
allows us to place restrictions on what $A_r$ can contain, if we
want to determine $A_r$ itself, we need a way to certify membership in it.
To certify inclusion in $A_r$ we need an upper bound on the defect, which
translates to an upper bound on complexity, which is relatively easy to do.
However we also need to discard $n$ that do not belong to $A_r$, i.e. pruning
the set we are starting with.
This requires establishing lower bounds on their defects, certifying they
are $r$ or larger,
and for this we need lower bounds on their complexities.

\subsection{Two pruning lemmas}\label{sec52}

To find lower bounds on
complexities, we typically use the following technique. Say we want to
show that $\cpx{n}\ge l$ ($l\in\mathbb{N}$); since $\cpx{n}$ is always an
integer, it suffices to show $\cpx{n}>l-1$. We do this by using our current knowledge
of $A_r$ for various $r$; by showing that if $\cpx{n}\le l-1$ held, then  it would put
$n$ in some $A_r$ which we have already determined and know it's not in. The
following two lemmas, both examples of this principle, are useful for this
purpose.

\begin{lem}
\label{multlem}
If $\alpha\le 1/2$, $i+j=k+2$, and $a$ and $b$ are natural numbers then
\[ a\in A_{i\alpha},\quad b\in A_{j\alpha},\quad ab\notin A_{k\alpha}
\quad \Longrightarrow \quad \cpx{ab}=\cpx{a}+\cpx{b}.\]
\end{lem}

\begin{proof}
Note
\[\cpx{ab}\ge3\log_3(ab)+k\alpha=3\log_3 a+3\log_3 b+(i+j-2)\alpha>
\cpx{a}+\cpx{b}-1\]
so $\cpx{ab}\ge \cpx{a}+\cpx{b}$.
\end{proof}

\begin{lem}
\label{addlem}
For natural numbers $a$, $k$, and $m\ge 0$ we have
\[ a\in A_{k\alpha}, \quad 3^m(a+1)\notin A_{k\alpha} \quad
\Longrightarrow \quad \cpx{3^m(a+1)}=\cpx{a}+3m+1.\]
\end{lem}

\begin{proof}
Note \[\cpx{3^m(a+1)}\ge3\log_3(a+1)+3m+k\alpha>\cpx{a}+3m\] so
$\cpx{3^m(a+1)}\ge3m+\cpx{a}+1$.
\end{proof}

In applying  the lemmas to verify that a given $n$ does not lie in a given
$A_r$,  one must check that $n$ is not in some other $A_s$.  In our
applications, we will have $s<r$, and $A_s$ will already be known, allowing the
required check.  In the following subsection we will typically not indicate
these checks explicitly, 
using the fact that in our cases one can always check whether $n\in A_s$ by
looking at the base-$3$ expansion of $n$.

\subsection{Proof of Theorem ~\ref{computeresult}: Inductive Steps }
\label{appcomp}

We prove Theorem~\ref{computeresult} by repeatedly
applying Theorem~\ref{themethod}, to go from $k$ to $k+1$ for $0 \le k \le 12$.
We will use a step size $\alpha=\dft(2)$, so let us first determine
$T_{\dft(2)}$. We compute that
$3<(3^{\frac{1-\dft(2)}{3}}-1)^{-1}+1<4$, and so $T_{\dft(2)}=\{1,2,3\}$.
We note that in all cases of attempting to determine
$B_{(k+1)\alpha}$ we are considering, we will have $(k+1)\alpha\le12\dft(2)$,
and so if $\cpx{b}<(k+1)\alpha+3\log_3 2$, then
\[\cpx{b}<12\dft(2)+3\log_3 2=3.179\ldots,\]
so $\cpx{b}\le 3$, which for $b$ solid implies $b=1$.

The base cases $B_{0} = \emptyset$ and $B_{\dft(2)}=\{ 3\}$ were handled in
Proposition \ref{basecase}.
We now treat the $B_{k \dft(2)}$ in increasing order.

\begin{prop}
\[B_{2\dft(2)}=B_{\dft(2)}\cup\{2\},\]
and the elements of $A_{2\dft(2)}$ have the complexities listed in
Theorem~\ref{computeresult}.
\end{prop}

\begin{proof}
By the main theorem,
\begin{eqnarray*}
B_{2\dft(2)}\setminus B_{\dft(2)} & \subseteq & \{1,2,6,9,27\} \cup \\
& & \{3\cdot3^n+1 : n\ge 0\}\cup\{3(3\cdot3^n+1) : n\ge 0\}.
\end{eqnarray*}
We can exclude $1$ because $\dft(1)=1$, and we can exclude $6$, $9$, and
$27$ as they are not leaders.  For $3^{n+1}+1$, Lemma~\ref{addlem} shows
$\cpx{3^{n+1}+1}=3(n+1)+1$, and thus $\dft(3^{n+1}+1)=1-3\log_3
(1+3^{-(n+1)})$, which allows us to check that none of these lie in
$A_{2\dft(2)}$.  We can exclude $3(3^{n+1}+1)$ since
Lemma~\ref{addlem} shows it has the same defect as $3^{n+1}+1$ (and so
therefore also is not a leader).  Finally, checking the complexity of $2\cdot
3^k$ can be done with Lemma~\ref{multlem}.  
\end{proof}

To make later computations easier, let us observe here that
$\dft(3^1+1)=\dft(4)=2\dft(2)$;
$6\dft(2)<\dft(3^2+1)=\dft(10)<7\dft(2)$;
$8\dft(2)<\dft(3^3+1)=\dft(28)<9\dft(2)$; and that for $n\ge4$,
$9\dft(2)<\dft(3^n+1)<10\dft(2)$.

In the above, for illustration, we explicitly considered and excluded $3$, $6$,
$9$, $27$, and $3(3^{n+1}+1)$, but henceforth we will simply not mention any
multiplications by $3$.  If $n=3a$ is a good factorization, $n$ cannot be a
leader (by definition), but if it is not a good factorization, we can by
Theorem~$\ref{themethod}$ ignore it.

\begin{prop}
\[B_{3\dft(2)}=B_{2\dft(2)}\cup\{4\},\]
and the elements of $A_{3\dft(2)}$ have the complexities listed in
Theorem~\ref{computeresult}.
\end{prop}

\begin{proof}
By the main theorem,
\begin{eqnarray*}
B_{3\dft(2)}\setminus B_{2\dft(2)} & \subseteq & \{1,4\} \cup \\
& & \{3\cdot3^n+1 : n\ge 0\}\cup\{2\cdot3^n+1 : n\ge 0\}.
\end{eqnarray*}
Again, $\dft(1)=1$.  By the above computation, the only number of the form
$3^{n+1}+1$ occuring in $A_{3\dft(2)}$ is $4$.  Lemma~\ref{addlem}
shows that $\cpx{2\cdot3^n+1}=3+3n$ for $n>0$, and hence
$\dft(2\cdot3^n+1)=3-3\log_3(2+3^{-n})$, which allows us to check that none of
these lie in $A_{3\dft(2)}$.  Finally, checking the complexity of
$4\cdot 3^k$ can be done with Lemma~\ref{multlem}.
\end{proof}

To make later computations easier, let us observe here that
$6\dft(2)<\dft(2\cdot3^1+1)=\dft(7)<7\dft(2)$;
$8\dft(2)<\dft(2\cdot3^2+1)=\dft(19)<9\dft(2)$;
$9\dft(2)<\dft(2\cdot3^3+1)=\dft(55)<10\dft(2)$; and that for $n\ge 4$,
$10\dft(2)<\dft(2\cdot3^n+1)<11\dft(2)$.

We will henceforth stop explicitly considering and then excluding $1$,
since we know that $9\dft(2)<\dft(1)=1<10\dft(2)$.
\begin{prop}
\[B_{4\dft(2)}=B_{3\dft(2)}\cup\{8\},\]
and the elements of $A_{4\dft(2)}$ have the complexities listed in
Theorem~\ref{computeresult}.
\end{prop}

\begin{proof}
By the main theorem,
\begin{eqnarray*}
B_{4\dft(2)}\setminus B_{3\dft(2)} & \subseteq & \{8\}\cup
\{3\cdot3^n+1 : n\ge 0\}\cup\\ & & \{2\cdot3^n+1 : n\ge 0\}\cup
\{4\cdot3^n+1 : n\ge 0\}.
\end{eqnarray*}
By the above computation, no numbers of the form $3^{n+1}+1$ or $2\cdot3^n+1$
occur in $A_{4\dft(2)}\setminus A_{3\dft(2)}$.  Lemma~\ref{addlem}
shows $\cpx{4\cdot3^n+1}=5+3n$ and hence
$\dft(4\cdot3^n+1)=5-3\log_3(4+3^{-n})$, which allows us to check that none of
these lie in $A_{4\dft(2)}$.  Finally, checking the complexity of
$8\cdot 3^k$ can be done with Lemma~\ref{multlem}.  
\end{proof}

To make later computations easier, let us observe here that
$5\dft(2)<\dft(4\cdot3^0+1)=\dft(5)<6\dft(2)$;
$9\dft(2)<\dft(4\cdot3^1+1)=\dft(13)<10\dft(2)$;
$10\dft(2)<\dft(4\cdot3^2+1)=\dft(37)<11\dft(2)$; and that for $n\ge 3$,
$11\dft(2)<\dft(4\cdot3^n+1)<12\dft(2)$.
\begin{prop}
\[B_{5\dft(2)}=B_{4\dft(2)}\cup\{16\},\]
and the elements of $A_{5\dft(2)}$ have the complexities listed in
Theorem~\ref{computeresult}.
\end{prop}

\begin{proof}
By the main theorem,
\begin{eqnarray*}
B_{5\dft(2)}\setminus B_{4\dft(2)} & \subseteq & \{16\}\cup
\{3\cdot3^n+1 : n\ge 0\}\cup\{2\cdot3^n+1 : n\ge 0\}\cup\\ & &
\{4\cdot3^n+1 : n\ge 0\}\cup\{8\cdot3^n+1 : n\ge 0\}.
\end{eqnarray*}
By the above computation, no numbers of the form $3^{n+1}+1$, $2\cdot3^n+1$, or
$4\cdot3^n+1$ occur in $A_{5\dft(2)}\setminus A_{4\dft(2)}$. Lemma~\ref{addlem}
shows that $\cpx{8\cdot3^n+1}=7+3n$ for $n>0$, and hence
$\dft(8\cdot3^n+1)=7-3\log_3(8+3^{-n})$, which allows us to check that none of
these lie in $A_{5\dft(2)}$.  Finally, checking the
complexity of $16\cdot 3^k$ can be done with Lemma~\ref{multlem}.  
\end{proof}

To make later computations easier, let us observe here that
$11\dft(2)<\dft(8\cdot3^1+1)=\dft(25)<\dft(8\cdot3^2+1)=\dft(73)<12\dft(2)$,
and that for $n\ge3$, $\dft(8\cdot3^n+1)>12\dft(2)$.
\begin{prop}
\[B_{6\dft(2)}=B_{5\dft(2)}\cup\{32,5\},\]
and the elements of $A_{6\dft(2)}$ have the complexities listed in
Theorem~\ref{computeresult}.
\end{prop}

\begin{proof}
By the main theorem,
\begin{eqnarray*}
B_{6\dft(2)}\setminus B_{5\dft(2)} & \subseteq & \{32\}\cup
\{3\cdot3^n+1 : n\ge 0\}\cup\{2\cdot3^n+1 : n\ge 0\}\cup\\ & &
\{4\cdot3^n+1 : n\ge 0\}\cup\{8\cdot3^n+1 : n\ge 0\}\cup\\ & &
\{16\cdot3^n+1 : n\ge 0\}.
\end{eqnarray*}
By the above computations, the number of any of the forms $3^{n+1}+1$,
$2\cdot3^n+1$, $4\cdot3^n+1$, or $8\cdot3^n+1$ occurring in
$A_{5\dft(2)}\setminus A_{4\dft(2)}$ is $5=4\cdot3^0+1$.
Lemma~\ref{addlem} shows that $\cpx{16\cdot3^n+1}=9+3n$, and hence
$\dft(16\cdot3^n+1)=9-3\log_3(16+3^{-n})$, which allows us to check that none
of these lie in $A_{6\dft(2)}$.  Finally, checking the
complexity of $32\cdot3^k$ can be done with Lemma~\ref{multlem}, and checking
the complexity of $5\cdot 3^k$ can be done with Lemma~\ref{addlem}.
\end{proof}

To make later computations easier, let us observe here that
$11\dft(2)<\dft(16\cdot3^0+1)=\dft(17)<12\dft(2)$, and that for $n\ge1$,
$\dft(16\cdot3^n+1)>12\dft(2)$.

In the above, for illustration, we explicitly considered and excluded numbers of
the form $3\cdot3^n+1$, $2\cdot3^n+1$, etc., for large $n$, despite having
already computed their complexities earlier.  Henceforth, to save space, we will
simply not consider a number if we have already computed its defect and seen it
to be too high.  E.g., in the above proof, we would have simply said, ``By the
main theorem and the above computations, $B_{6\dft(2)}\setminus B_{5\dft(2)}
\subseteq \{32,5\}\cup\{8\cdot3^n+1 : n\ge 0\}$''.
\begin{prop}
\[B_{7\dft(2)}=B_{6\dft(2)}\cup\{64,7,10\},\]
and the elements of $A_{7\dft(2)}$ have the complexities listed in
Theorem~\ref{computeresult}.
\end{prop}

\begin{proof}
By the main theorem and the above computations,
\[B_{7\dft(2)}\setminus B_{6\dft(2)} \subseteq \{64,7,10\}\cup
\{32\cdot3^n+1 : n\ge 0\}\cup\{5\cdot3^n+1 : n\ge 0\}.\]
Lemma~\ref{addlem} shows that $\cpx{32\cdot3^n+1}=11+3n$
and, for $n\ge 2$, $\cpx{5\cdot3^n+1}=6+3n$. Hence
$\dft(32\cdot3^n+1)=11-3\log_3(32+3^{-n})$, and, for $n\ge 2$,
$\dft(5\cdot3^n+1)=6-3\log_3(5+3^{-n})$ which allows us to check that none of
these lie in $A_{7\dft(2)}$.
Finally, checking the complexities of $64\cdot3^k$, $7\cdot3^k$, and
$10\cdot3^k$ can be done via Lemma~\ref{multlem} (for $64$ and $10$) and
Lemma~\ref{addlem} (for $7$ and $10$).
\end{proof}

To make later computations easier, let us observe here that
$\dft(32\cdot3^n+1)>12\dft(2)$ for all $n$, and that for $n\ge2$,
$\dft(5\cdot3^n+1)>12\dft(2)$ as well.  Indeed, as we will see, from this
point on, no new examples of multiplying by a power of $3$ and then adding $1$
will ever have complexity less than $12\dft(2)$.
\begin{prop}
\[B_{8\dft(2)}=B_{7\dft(2)}\cup\{128,14,20\},\]
and the elements of $A_{8\dft(2)}$ have the complexities listed in
Theorem~\ref{computeresult}.
\end{prop}

\begin{proof}
By the main theorem and the above computations,
\begin{eqnarray*}
B_{8\dft(2)}\setminus B_{7\dft(2)} & \subseteq & \{128,14,20\}\cup
\{64\cdot3^n+1 : n\ge 0\}\cup\\ & & \{7\cdot3^n+1 : n\ge 0\}\cup
\{10\cdot3^n+1 : n\ge 0\}.
\end{eqnarray*}
Lemma~\ref{addlem} shows that
$\cpx{64\cdot3^n+1}=13+3n$, $\cpx{10\cdot3^n+1}=8+3n$, and, for $n\ne 0, 2$,
$\cpx{7\cdot3^n+1}=7+3n$.  Using this to check their defects, we see that none
of these lie in $A_{8\dft(2)}$, or even $A_{12\dft(2)}$.
Finally, checking the complexities of $128\cdot3^k$, $14\cdot3^k$, and
$20\cdot3^k$ can be done with Lemma~\ref{multlem}.
\end{proof}
\begin{prop}
\[B_{9\dft(2)}=B_{8\dft(2)}\cup\{256,28,40,19\},\]
and the elements of $A_{9\dft(2)}$ have the complexities listed in
Theorem~\ref{computeresult}.
\end{prop}

\begin{proof}
By the main theorem and the above computations,
\begin{eqnarray*}
B_{9\dft(2)}\setminus B_{8\dft(2)} & \subseteq & \{256,28,40,19\}\cup
\{128\cdot3^n+1 : n\ge 0\}\cup\\ & &\{14\cdot3^n+1 : n\ge 0\}\cup
\{20\cdot3^n+1 : n\ge 0\}.
\end{eqnarray*}
Lemma~\ref{addlem} shows that
$\cpx{128\cdot3^n+1}=15+3n$, and for $n\ge 1$, $\cpx{14\cdot3^n+1}=9+3n$ and
$\cpx{20\cdot3^n+1}=10+3n$.  Using this to check their defects, we see that none
of these lie in $A_{8\dft(2)}$, or even $A_{12\dft(2)}$.
Finally, checking the complexities of $256\cdot3^k$, $28\cdot3^k$, and
$40\cdot3^k$, and $19\cdot3^k$ can be done via Lemma~\ref{multlem} (for $256$,
$28$, and $40$) and Lemma~\ref{addlem} (for $28$ and $19$).
\end{proof}
\begin{prop}\label{level10}
\[B_{10\dft(2)}=B_{9\dft(2)}\cup\{512,13,1,56,80,55,38\}\cup
\{3\cdot3^n+1:n\ge 3\},\]
and the elements of $A_{10\dft(2)}$ have the complexities listed in
Theorem~\ref{computeresult}.
\end{prop}

\begin{proof}
By the main theorem and the above computations,
\begin{eqnarray*}
B_{10\dft(2)}\setminus B_{9\dft(2)} & \subseteq & \{512,13,1,56,80,55,38\}
\cup \{3\cdot3^n+1:n\ge 3\}\cup \\ & &
\{256\cdot3^n+1 : n\ge 0\}\cup\{28\cdot3^n+1 : n\ge 0\}\cup\\ & &
\{40\cdot3^n+1 : n\ge 0\}\cup\{19\cdot3^n+1 : n\ge 0\}.
\end{eqnarray*}
We know $\dft(1)=1$.  Lemma~\ref{addlem} shows that $\cpx{256\cdot3^n+1}=17+3n$,
$\cpx{28\cdot3^n+1}=11+3n$, $\cpx{40\cdot3^n+1}=12+3n$, and for $n\ge 1$,
$\cpx{19\cdot3^n+1}=10+3n$.  Using this to check their defects, we see that none
of these lie in $A_{10\dft(2)}$, or even $A_{12\dft(2)}$.  Finally, checking the
complexities of $512\cdot3^k$, $13\cdot3^k$, $56\cdot3^k$, $80\cdot3^k$,
$55\cdot3^k$, $38\cdot3^k$, and $(3^{n+1}+1)3^k$ can be done via
Lemma~\ref{multlem} (for $512$, $56$, $80$, and $38$) and Lemma~\ref{addlem}
(for $13$, $55$ and $3^{n+1}+1$).
\end{proof}
\begin{prop}
\begin{eqnarray*}
B_{11\dft(2)} & = & B_{10\dft(2)}\cup\{1024,26,112,37,160,110,76\}\cup\\
& &\{2(3\cdot3^n+1):n\ge 3\}\cup\{2\cdot3^n+1:n\ge 4\},
\end{eqnarray*}
and the elements of $A_{11\dft(2)}$ have the complexities listed in
Theorem~\ref{computeresult}.
\end{prop}

\begin{proof}
By the main theorem and the above computations,
\begin{eqnarray*}
B_{11\dft(2)}\setminus B_{10\dft(2)} & \subseteq &
\{1024,26,112,37,160,110,76,25\} \cup \\ & & \{2(3\cdot3^n+1):n\ge 3\}\cup
\{2\cdot3^n+1:n\ge4\} \cup \\ & & \{512\cdot3^n+1 : n\ge 0\}\cup
\{13\cdot3^n+1 : n\ge 0\} \cup \\ & & \{56\cdot3^n+1 : n\ge 0\}\cup
\{80\cdot3^n+1 : n\ge 0\} \cup \\ & & \{55\cdot3^n+1 : n\ge 0\}\cup
\{38\cdot3^n+1 : n\ge 0\}\cup \\ & & \{(3\cdot3^n+1)3^m+1: n\ge3, m\ge 0\}
\end{eqnarray*}
Lemma~\ref{addlem} shows that for $m\ge3$,
$\cpx{(3^{m+1}+1)3^n+1}=2+3(m+1)+3n$, and that for $n\ge 1$,
$\cpx{512\cdot3^n+1}=19+3n$, $\cpx{56\cdot3^n+1}=13+3n$,
$\cpx{80\cdot3^n+1}=14+3n$, $\cpx{55\cdot3^n+1}=13+3n$,
$\cpx{38\cdot3^n+1}=12+3n$, and that for $n\ge 2$, $\cpx{13\cdot3^n+1}=9+3n$.
Using this to check their defects, we see that none of these lie in
$A_{11\dft(2)}$, or even $A_{12\dft(2)}$.
We checked earlier that $\dft(25)>11\dft(2)$.
Finally, checking the complexities of $1024\cdot3^k$, $26\cdot3^k$,
$112\cdot3^k$, $37\cdot3^k$, $160\cdot3^k$, $110\cdot3^k$, $76\cdot3^k$,
$2(3^{n+1}+1)3^k$, and $(2\cdot3^n+1)3^k$ can be done via Lemma~\ref{multlem}
(for $1024$, $26$, $112$, $160$, $110$, $76$, and $2(3^{n+1}+1)$) and
Lemma~\ref{addlem} (for $37$ and $2\cdot3^n+1$).
\end{proof}
\begin{prop}
\begin{eqnarray*}
B_{12\dft(2)} & = & B_{11\dft(2)}\cup\{2048,25,52,224,74,320,17,220,152,73\}
\cup\\ & & \{4(3\cdot3^n+1):n\ge 3\}\cup\{2(2\cdot3^n+1):n\ge4\} \cup \\
& & \{4\cdot3^n+1:n\ge3\}
\end{eqnarray*}
and the elements of $A_{12\dft(2)}$ have the complexities listed in
Theorem~\ref{computeresult}.
\end{prop}

\begin{proof}
By the main theorem and the above computations,
\begin{eqnarray*}
B_{12\dft(2)}\setminus B_{11\dft(2)} & \subseteq &
\{2048,25,52,224,74,320,17,220,152,73,35\} \cup \\ & &
\{4(3\cdot3^n+1):n\ge 3\}\cup \{2(2\cdot3^n+1):n\ge4\}\cup
\\ & & \{4\cdot3^n+1:n\ge3\}
\cup\{1024\cdot3^n+1 : n\ge 0\}\cup\\ & &\{26\cdot3^n+1 : n\ge 0\}
\cup\{112\cdot3^n+1 : n\ge 0\}\cup \\ & & \{37\cdot3^n+1:n\ge 0\}
\cup\{160\cdot3^n+1 : n\ge 0\}\cup \\ & &
\{110\cdot3^n+1 : n\ge 0\}\cup \{76\cdot3^n+1 : n\ge 0\}\cup \\ & &
\{2(3\cdot3^n+1)3^m+1: n\ge3, m\ge 0\}\cup \\ & &
\{(2\cdot3^n+1)3^m+1: n\ge4, m\ge 0\}
\end{eqnarray*}
Lemma~\ref{addlem} shows that for
$m\ge3$ and $n\ge1$, $\cpx{2(3^{m+1}+1)3^n+1}=4+3(m+1)+3n$, and that for $m\ge
4$ and $n\ge 1$, $\cpx{(2\cdot3^m+1)3^n+1}=4+3m+3n$, and that
$\cpx{1024\cdot3^n+1}=21+3n$, $\cpx{112\cdot3^n+1}=15+3n$,
$\cpx{160\cdot3^n+1}=16+3n$, $\cpx{76\cdot3^n+1}=14+3n$, and that for $n\ge 1$,
$\cpx{26\cdot3^n+1}=11+3n$, $\cpx{110\cdot3^n+1}=15+3n$, and that for $n\ge 2$,
$\cpx{37\cdot3^n+1}=12+3n$.  Using this to check their defects, we see that none
of these lie in $A_{12\dft(2)}$.
We can then check that $\dft(35)>12\dft(2)$.
Finally, checking the complexities of $2048\cdot3^k$, $25\cdot3^k$,
$52\cdot3^k$, $224\cdot3^k$, $74\cdot3^k$, $320\cdot3^k$, $220\cdot3^k$,
$152\cdot3^k$, $73\cdot3^k$, $4(3^{n+1}+1)3^k$, $2(2\cdot3^n+1)3^k$, and
$(4\cdot3^n+1)3^k$ can be done via Lemma~\ref{multlem} (for $2048$, $25$, $52$,
$224$, $74$, $320$, $220$, $152$, $4(3^{n+1}+1)$, and $2(2\cdot3^n+1)$) and
Lemma~\ref{addlem} (for $25$, $17$, $73$, and $4\cdot3^n+1$).
\end{proof}

Combining all these propositions establishes Theorem~\ref{computeresult}.

\section{Applications}
\label{theory}

We now present several applications of the classification obtained
in Section 5. These are:
(i) Stability of numbers $n>1$ of defect less than $12\dft(2)+1$;
(ii) Classification of all integers $n$ having defect $0 \le \dft(n) \le 1$
and finiteness of $B_{r}$ for all $r<1$;
(iii) Determination of complexities $\cpx{2^{a}\cdot 3^k}$ for $a \le 21$
and all $k$;
(iv) Upper bounds on the number of integers $n \le x$ having complexity
$\dft(n) < r$, for any fixed $r>0$.

\subsection{Stability of numbers of low defect}

We have already noted in Theorem~\ref{computeresult} that numbers $n>1$ of
defect less than $12\dft(2)$ are stable.  In fact, we can conclude something
stronger.

\begin{thm}
If $n>1$ and $\dft(n)<12\dft(2)+1=2.2865\ldots$, then $n$ is stable.
\end{thm}

\begin{proof}
 From Theorem~\ref{computeresult}, we can check that if $\dft(3n)<12\dft(2)$,
then $\dft(n)<12\dft(2)$.  So suppose the theorem were false, and we have
unstable $n>1$ with $\dft(n)<12\dft(2)+1$.  Then for some $K$, $\dft(3^K n)\le
\dft(n)-1<12\dft(2)$.  So by above, we have $\dft(n)<12\dft(2)$, and thus, as
noted in Theorem~\ref{computeresult}, $n$ is stable unless $n=1$.
\end{proof}

In fact, if $n>1$ and $\dft(n)<\dft(107)=3.2398\ldots$, then $n$ is stable,
as we will prove in \cite{seq2}.

\subsection{Classifying the integers of Defect at most $1$}
\label{sec61}

   Using  Theorem~\ref{computeresult} we can classify all the numbers with defect
less than $1$, as follows:

\begin{thm}
The natural numbers $n$ satisfying $\dft(n)<1$ are precisely those that can be
written in one of the following forms, and have the following complexities:
\begin{enumerate}
\item $3^k$ for $k\ge 1$, of complexity $3k$
\item $2^a 3^k$ for $a\le 9$, of complexity $2a+3k$ (for $a$, $k$ not both
zero)
\item $5\cdot2^a 3^k$ for $a\le 3$, of complexity $5+2a+3k$
\item $7\cdot2^a 3^k$ for $a\le 2$, of complexity $6+2a+3k$
\item $19\cdot3^k$ of complexity $9+3k$
\item $13\cdot3^k$ of complexity $8+3k$
\item $(3^n+1)3^k$ of complexity $1+3n+3k$ (for $n\ne0$)
\end{enumerate}
Furthermore $n=1$ is the only number having defect exactly $1$.
\end{thm}

\begin{proof}
This list includes all numbers in $A_{9 \dft(2)}$,
and some numbers in
$A_{10\dft(2)}$. These in turn are determined by the
corresponding lists for $B_{9 \dft(2)}, B_{10 \dft(2)}$,
in the latter case (Proposition \ref{level10}) checking the complexities to
exclude the leaders
$\{56, 80, 55, 38\}$.
\end{proof}
Using  this list one
may deduce the following important fact.
\begin{thm}
\label{finite}
For every $0< \alpha<1$, the set of leaders $B_\alpha$ is a finite set.
For every $\alpha \ge 1$, the set $B_{\alpha}$ is an infinite set.
\end{thm}

\begin{proof}
The first part follows from the fact that each of the categories above has a
finite set of leaders, and that the final list (7) has a finite number of
sublists
with defect smaller than $1- \epsilon$, for any epsilon.
The defects
$$
\dft((3^n+1)3^k) =
(3n+1) - 3 \log_3(3^n+1)
= 1-3 \log_3(1+ \frac{1}{3^n})
$$
approach $1$
from below as $n$
approaches infinity. This also establishes that $B_{1}$ is an
infinite set, giving the second part.
\end{proof}

\subsection{The complexity of $2^m3^k$ for small $m$}
\label{sec62}

The determination of $A_r$ in
Theorem \ref{computeresult}
allows us to put lower bounds on the complexities of any
numbers not in it.  Thus for instance we have the following result.

\begin{lem}
Let $n$ be a natural number and suppose that there is no $k$ such that
$2^{n+9}3^k\in A_{n\dft(2)}$.  Then for any $m\le n+9$ and any $k$ (with $m$
and $k$ not both zero), $\cpx{2^m 3^k}=2m+3k$.
\end{lem}

\begin{proof}
It suffices to show that $\cpx{2^{n+9}3^k}>2n+3k+17$, but by assumption,
\[\cpx{2^{n+9}3^k}\ge(n+9)3\log_3 2+3k+n\dft(2)=2n+3k+27\log_3 2>2n+3k+17,\]
and we are done.
\end{proof}

This lemma immediately establishes Conjecture \ref{cj11} for $ a \le 21$.

\begin{proof}[Proof of Theorem \ref{th11main}.] From our classification, it is
straightforward to check that $2^{21} 3^k$
does not lie in $A_{12\dft(2)}$ for any $k$, so we can conclude that
for $m\le 21$ and any $k$, with $m$ and $k$ not both zero, $\cpx{2^m
3^k}=2m+3k$.
\end{proof}

\subsection{Counting the integers below $x$ having defect at most $r$}
\label{sec63}

In our computations in Section \ref{sec5}, we used a small step size
$\alpha=\dft(2)$, and kept our superset of $A_r$ small by using a pruning
step. In what follows, we will use a different trick to keep our supersets of
$A_r$ from getting too large. Instead of pruning, we will use step sizes
arbitrarily close to $1$.

\begin{prop}
\label{indcount}
Given any $0<\alpha<1$, and any $k\ge1$, we have that
$B_{k\alpha}(x)=O_{k\alpha}((\log x)^{k-1})$, and
$A_{k\alpha}(x)=O_{k\alpha}((\log x)^k)$.
\end{prop}

\begin{proof}
We induct on $k$.  Suppose $k=1$; by Corollary \ref{finite}, then
$B_{k\alpha}=B_\alpha$ is a finite set, so $B_{k\alpha}(x)=O_{k\alpha}(1)$.
Also, for any $r$, $A_r(x)\le B_r(x)(\log_3 x)$; in particular,
$A_{k\alpha}(x)=O_{k\alpha}(\log x)$.

So suppose it is true for $k$ and we want to prove it for $k+1$; we apply
Proposition~\ref{themethod} with step size $\alpha$.  For convenience, let $S_r$
denote the set of solid numbers $b$ satisfying $\cpx{b}<r+3\log_3 2$, as
mentioned in the discussion after Theorem~\ref{themethod}; for any $r$, this is
a finite set.

In the case $k+1=2$,
\begin{eqnarray*}
B_{2\alpha}(x) & \le & B_\alpha(x)^3 +
(A_\alpha(x)|S_{2\alpha}| + |T_\alpha|)(|B_\alpha|+1) \\
& = & O_{\alpha}(1)^3 + O_\alpha(\log x) + O_\alpha(1) \\
& = & O_{(k+1)\alpha}(\log x).
\end{eqnarray*}
In the case $k+1>2$,
\small
\begin{eqnarray*}
B_{(k+1)\alpha}(x) & \le& \sum_{\substack{i+j=k+2 \\ i,j\ge2}} B_{i\alpha}(x)
B_{j\alpha}(x) + (A_{k\alpha}(x)|S_{(k+1)\alpha}| +
|T_\alpha|)(|B_\alpha|+1) \\
& =& \sum_{\substack{i+j=k+2 \\ i,j\ge2}} O_{i\alpha}((\log x)^{i-1})
O_{j\alpha}((\log x)^{j-1}) + O_{(k+1)\alpha}((\log x)^k) + O_\alpha(1) \\
& =& O_{k\alpha}((\log x)^k).
\end{eqnarray*}
\normalsize

In either case, we also have $A_{(k+1)\alpha}(x)=O_{(k+1)\alpha}((\log
x)^{k+1})$.  This completes the proof.

\end{proof}

Using this result we conclude:
\begin{thm}\label{upperbound}
For any number $r>0$, $B_r(x)=\Theta_r((\log x)^{\lfloor r \rfloor})$, and
$A_r(x)=\Theta_r((\log x)^{\lfloor r \rfloor+1})$.
\end{thm}

\begin{proof}
For the upper bound, it suffices to note that
$r=(\lfloor r \rfloor +1 )\frac{r}{\lfloor r \rfloor+1}$, and that
$\frac{r}{\lfloor r \rfloor+1}<1$,
and apply Proposition~\ref{indcount}.

For the lower bound, let $k=\lfloor r \rfloor$, and consider numbers of the
form
\[N=((\cdots((3\cdot3^{n_k}+1)3^{n_{k-1}}+1)\cdots)3^{n_1}+1)3^{n_0}.\]
Then
\[\cpx{N}\le3(n_0+\cdots+n_k+1)+k\]
and since $\log_3 N\ge n_0+\cdots+n_k+1$, this means $\dft(N)\le k$.
Furthermore, if $n_0=0$ and $n_1>0$ then $N$ is not divisible by $3$ and so is a
leader.  It is then easy to count that there are at least $\binom{\lfloor \log_3
x \rfloor}{k+1}\gtrsim \frac{1}{(k+1)!}(\log_3 x)^{k+1}$ such $N$ less than a
given $x$, and at least $\binom{\lfloor \log_3 x \rfloor}{k}\gtrsim
\frac{1}{k!}(\log_3 x)^k$ if we insist that $N$ be a leader.
\end{proof}

An immediate consequence of Theorem \ref{upperbound} is Theorem~\ref{indcount0}
in the introduction.

\begin{proof}[Proof of Theorem  \ref{indcount0}.]
The existence of numbers of arbitrarily large defect follows from the fact
that the set of integers of defect $< r$ has density zero.
\end{proof}

This result is a long way from proving a bound of the type $\cpx{n}\nsim 3\log_3
n$.

\section{Acknowledgements}

The authors are indebted to
J\=anis Iraids and Karlis Podnieks for supplying a wealth of numerical data.
We thank Jeffrey Lagarias for looking over an early draft of this paper and
elucidating just what it was we were doing, as well as for other help with
editing, and to Mike Zieve and David Rohrlich for providing assistance with
early drafts of the paper.  We thank Paul Pollack and Mike Bennett for pointing
out the paper \cite{Stewart}.
Most of all we thank Juan Arias de Reyna for greatly clarifying much of our
work, suggesting improved notation, shortening some proofs, and helping
extensively with structuring and editing of this paper. We thank the reviewer
for very helpful comments.

\appendix


\begin{thebibliography}{99}
\bibitem{seq3} Harry Altman, Integer Complexity: The Integer Defect, in
preparation.
\bibitem{seq1} Harry Altman, Integer Complexity and Well-Ordering, in
preparation.
\bibitem{seq2} Harry Altman, Computation of Numbers with Integer Complexity
Close to the Lower Bound, in preparation.
\bibitem{Arias} J.~Arias de Reyna, Complejidad de los n\'umeros naturales,
{\it Gaceta R.\ Soc.\ Mat.\ Esp.}, {\bf3} (2000), 230--250.
\bibitem{Guy} Richard K.~Guy, 
Some suspiciously simple sequences,
 {\it Amer.\ Math.\ Monthly}, {\bf 93} (1986), 186--190;
 and see {\bf 94} (1987), 965 \& {\bf 96} (1989), 905.
\bibitem{UPINT} Richard K.~Guy, {\it Unsolved Problems in Number Theory},
Third Edition, Springer-Verlag, New York, 2004, pp. 399--400.
\bibitem{data2} J\=anis Iraids, Kaspars Balodis, Juris \v{C}er\c{n}enoks,
M\=arti\c{n}\v{s} Opmanis, Rihards Opmanis, K\=arlis Podnieks.
Integer Complexity: Experimental and Analytical results, arXiv:1203.6462, 2012
\bibitem{Lag09}
Jeffrey~C.~Lagarias, On ternary expansions of powers of $2$,
{\it J. London Math. Soc.}, {\bf 79} (2009), 562--588. {\it MR} 2508867.
\bibitem{MP} K.~Mahler \& J.~Popken, On a maximum problem in arithmetic
 (Dutch), {\it Nieuw Arch.\ Wiskunde}, (3) {\bf 1} (1953),
 1--15; {\it MR} {\bf 14}, 852e.
\bibitem{Raws} Daniel A.~Rawsthorne, How many 1's are needed?, {\it Fibonacci
Quart.}, {\bf27} (1989), 14--17; {\it MR} {\bf90b}:11008.
\bibitem{waset} Vivek V. Srinivas  \& B.~R.~Shankar,  Integer Complexity:
Breaking the $\Theta(n^2)$ barrier, {\it World Academy of Science}, {\bf
41} (2008), 690--691
\bibitem{Stewart} C.~L.~Stewart, On the Representation of an Integer in Two
Different Bases, {\it J. Reine Angew. Math.}, {\bf 319} (1980), 63--72. 
\bibitem{upbds} Joshua Zelinsky, An Upper Bound on Integer Complexity, in
preparation
\end{thebibliography}
\end{document}